\documentclass[graybox]{svmult}

\usepackage{type1cm}        
\usepackage{makeidx}         
\usepackage{graphicx}        
\usepackage{multicol}        
\usepackage[bottom]{footmisc}
\usepackage{tikz}

\usepackage{newtxtext}       %
\usepackage[varvw]{newtxmath}       
\newcommand{\rightt}{\right]\right]}
\newcommand{\leftt}{\left[\left[}
\newcommand{\opt}{\text{opt}}
\makeindex             

\begin{document}

\title*{Weak scalability of domain decomposition methods for discrete fracture networks}
\author{Stefano Berrone and Tommaso Vanzan}
\institute{Stefano Berrone \at Politecnico di Torino, Italy, \email{stefano.berrone@polito.it}
\and Tommaso Vanzan \at Ecole Polytecnique F\'{e}d\'{e}rale de Lausanne, Switzerland,  \email{tommaso.vanzan@epfl.ch}}

\maketitle

\abstract*{Discrete Fracture Networks (DFNs) are complex three-dimensional structures 
characterized by the intersection of planar polygonal fractures, and are used to model flows in fractured media. Despite being suitable for Domain Decomposition (DD) techniques, there are relatively few works on the application of DD methods to DFNs, see, e.g., \cite{pichot2014mortar,doi:10.1137/120865884} and references therein.
In this manuscript, we perform a theoretical study of Optimized Schwarz Methods (OSMs) applied to DFNs.
Interestingly, we prove that the OSMs can be weakly scalable (that is, they may converge to a given tolerance in a number of iterations independent of the number of fractures) under suitable assumptions on the domain decomposition. This contribution fits in the renewed interest on the weak scalability of DD methods after the works \cite{cances2013domain,ciaramella2020scalability,Chaouqui2018}, which showed weak scalability of DD methods for specific geometric configurations, even without coarse spaces.
Despite simplification assumptions which may be violated in practice, our analysis provides heuristics to minimize the computational efforts in realistic settings.
Finally, we emphasize that the methodology proposed can be straightforwadly generalized to study other classical DD methods applied to DFNs (see, e.g., \cite{Chaouqui2018}).}


\section{Introduction}
Discrete Fracture Networks (DFNs) are complex three-dimensional structures 
characterized by the intersections of planar polygonal fractures, and are used to model flows in fractured media. Despite being suitable for Domain Decomposition (DD) techniques, there are relatively few works on the application of DD methods to DFNs, see, e.g., \cite{pichot2014mortar,doi:10.1137/120865884} and references therein.

In this manuscript, we present a theoretical study of Optimized Schwarz Methods (OSMs) applied to DFNs.
Interestingly, we prove that the OSMs can be weakly scalable (that is, they converge to a given tolerance in a number of iterations independent of the number of fractures) under suitable assumptions on the domain decomposition. This contribution fits in the renewed interest on the weak scalability of DD methods after the works \cite{cances2013domain,ciaramella2020scalability,Chaouqui2018}, which showed weak scalability of DD methods for specific geometric configurations, even without coarse spaces.

Despite simplifying assumptions which may be violated in practice, our analysis provides heuristics to minimize the computational efforts in realistic settings.
Finally, we emphasize that the methodology proposed can be straightforwadly generalized to study other classical DD methods applied to DFNs (see, e.g., \cite{Chaouqui2018}).

\section{Scalability analysis for one-dimensional DFNs}
We start considering a simplified DFN made of one-dimensional fractures $F_i$, $i=1,...,N$ arranged in a staircase fashion depicted in Fig \ref{DFN}. The DFN is $\Omega:=\cup_{i=1}^{N} F_i$. The boundary of the fractures is denoted with $\partial F_i$ and it holds $\partial \Omega= \cup_{i=1}^{N}\partial F_i$. Further, $\partial \Omega$ can be decomposed into a Dirichlet boundary $\Gamma_D$ and a Neumann boundary $\Gamma_N$, so that  $\partial \Omega=\Gamma_D \cup \Gamma_N$.
The intersections between fractures are called traces and are denoted by $S_m$, $m=1,...,N-1=:M$.
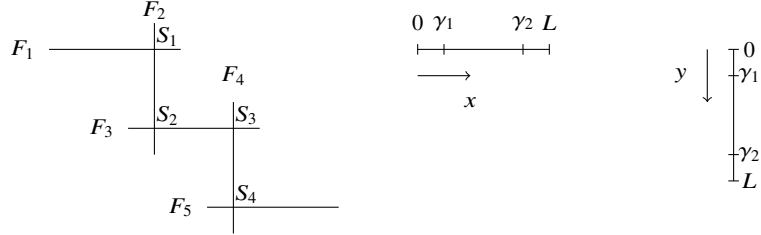
\begin{figure}\label{fig:1Dstair}
\centering
\begin{tikzpicture}[scale=0.35]
\node at (-26,10) {$F_1$};
\draw[black] (-25,10) -- (-20,10);
\draw[black] (-21,11) -- (-21,6);
\node at (-20.5,10.5) {$S_1$};
\node at (-21,11.5) {$F_2$};
\node at (-23,7) {$F_3$};
\node at (-20.5,7.5) {$S_2$};
\draw[black] (-22,7) -- (-17,7);
\node at (-18,9) {$F_4$};
\draw[black] (-18,8) -- (-18,3);
\node at (-17.5,7.5) {$S_3$};
\node at (-20,4) {$F_5$};
\draw[black] (-19,4) -- (-14,4);
\node at (-17.5,4.5) {$S_4$};
\begin{scope}[xshift={-6cm}]
\draw[black] (-5,10) -- (0,10);
\draw[black] (-5,9.8) -- (-5,10.2);
\draw[black] (0,9.8) -- (0,10.2);
\draw[black] (-1,9.8) -- (-1,10.2);
\draw[black] (-4,9.8) -- (-4,10.2);
\node at (-5,11) {$0$};
\draw[->] (-5,9) -- (-3,9);
\node at (-3,8) {$x$};
\node at (0,11) {$L$};
\node at (-4,11) {$\gamma_1$};
\node at (-1,11) {$\gamma_2$};
\quad 
\draw[black] (7,10) -- (7,5);
\draw[black] (6.8,10) -- (7.2,10);
\draw[black] (6.8,9) -- (7.2,9);
\draw[black] (6.8,5) -- (7.2,5);
\draw[black] (6.8,6) -- (7.2,6);
\draw[->] (6,10) -- (6,8);
\node at (5,9) {$y$};
\node at (7.6,10) {$0$};
\node at (7.6,5) {$L$};
\node at (7.6,9) {$\gamma_1$};
\node at (7.6,6) {$\gamma_2$};
\end{scope}
\end{tikzpicture}
\caption{Geometry of the simplified DFN and of its one-dimensional fractures.}\label{DFN}
\end{figure} 
We assume that both the vertical and horizontal fractures have two traces located at $\tau=\gamma_1$ and $\tau=\gamma_2$ with $\gamma_1<\gamma_2$, ($\tau$ being the local coordinate), except the first and last fracture.
The mathematical DFN model consists in the coupled system of partial differential equations for the hydraulic heads $u_j$,
\begin{align}
-\nu_j\partial_{\tau_j \tau_j} u_j&= f\quad \text{ in }F_j,\quad \mathcal{B}_j(u)=0\quad\text{ on } \partial F_j,\quad j=1,\dots,N,\label{eq:main_equation}\\
u_{|F_i}&=u_{|F_{i+1}}\quad \text{ on } S_{i},\text{ }i=1,...,M, \label{eq:continuity_condition}\\
\left[ \left[ \frac{\partial u_i}{\partial \tau_i}\right]\right] &+ \left[ \left[ \frac{\partial u_{i+1}}{\partial \tau_{i+1}}\right]\right]=0\quad \text{ on } S_{i},\text{ }i=1,...,M,\label{eq:flux_condition}
\end{align}
where $\mathcal{B}_j$ represent boundary conditions (b.c) (specified later), $\nu_j$ is the local diffusion coefficient, and $\left[\left[v\right]\right]$ is the jump of $v$ across the intersection of fractures. The local solutions $u_j$ are coupled through \eqref{eq:continuity_condition}-\eqref{eq:flux_condition} which enforce continuity of the hydraulic heads, and balance between the jumps of the conormal derivatives across the traces.

System \eqref{eq:main_equation}-\eqref{eq:flux_condition} is clearly prone to a DD approach. We consider a nonoverlapping domain decomposition in which each subdomain corresponds to a single fracture, and the 
optimized Schwarz method (OSM) that, starting from an initial guess $u^0_j$, $j=1,...,N$, computes for $n=1,2,...$ until convergence
\begin{align}\label{eq:OSM}
-\nu_j\partial_{\tau_j\tau_j} u^n_j&= f_j\quad \text{ in }F_j,\quad \mathcal{B}_j(u^n_j)=0\quad \text{ on } \partial F_i, \nonumber\\
\leftt \frac{\partial u^n_j}{\partial \tau_j}\rightt &+ s_{j-1}^+ u^n_{j} = -\leftt \frac{\partial u^{n-1}_{j-1}}{\partial \tau_{j-1}}\rightt + s_{j-1}^+ u^{n-1}_{j-1}\text{ on } S_{j-1}, \\
\leftt \frac{\partial u^n_j}{\partial \tau_j}\rightt &+ s_{j}^- u^n_{j} = -\leftt \frac{\partial u^{n-1}_{j+1}}{\partial \tau_{j+1}}\rightt + s_{j}^- u^{n-1}_{j+1} \text{ on } S_{j}.\nonumber
\end{align}
for $j=2,...,N-1$, while for $j=1,N$,
\begin{align}\label{eq:OSM_2}
&-\nu_1\partial_{\tau_1\tau_1} u^n_1= f_1\text{ in } F_1,\quad \mathcal{B}_1(u^n_1)=0,\quad
-\nu_N\partial_{\tau_{N}\tau_{N}} u^n_{N}= f_N \text{ in }F_{N}, \quad \mathcal{B}_{N}(u^n_{N})=0,\nonumber\\
&\leftt \frac{\partial u^n_1}{\partial \tau_1}\rightt + s_{1}^- u^n_{1} = -\leftt \frac{\partial u^{n-1}_{2}}{\partial \tau_{2}}\rightt + s_{1}^- u^{n-1}_{2} \text{ on } S_{1},\\
&\leftt \frac{\partial u^{n}_{N}}{\partial \tau_{N}}\rightt + s_{N-1}^+ u^n_{N} = -\leftt \frac{\partial u^{n-1}_{N-1}}{\partial \tau_{N-1}}\rightt + s_{N-1}^+ u^{n-1}_{N-1}\text{ on } S_{N-1}.\nonumber
\end{align}
The functions $f_j$ are the restriction of the force term on the fracture $F_j$ and $s_j^{+,-}$, $j=1,...,M$ are positive parameters. 

To carry out the scalability analysis, we assume for the sake of simplicity that $s_{j}^{+,-}=p \in \mathbb{R}^+$ and $\nu_j=1$ for all $j$. We study later how to optimize the choice for $s_j^{+,-}$. 
We first discuss the case in which every $\mathcal{B}_j$ represents a Dirichlet boundary condition, and then we treat the case in which Neumann b.c. are imposed everywhere, expect at the left boundary of $F_1$ (source fracture) and at the right boundary of $F_N$. More general configurations can be included straightforwardly in our analysis.

Due to the linearity of the problem, we define the errors $e_j^n:=u-u_j^n$ and study their convergence to zero. The errors $e_j$ satisfy an error system which is obtained setting $f_j=0$ in \eqref{eq:OSM}-\eqref{eq:OSM_2}. Inside each fracture, $e_j$ is harmonic and has the analytical expression 
\begin{align}\label{eq:expressions}
\medmuskip=-1mu
\thinmuskip=-1mu
\thickmuskip=-1mu
\nulldelimiterspace=0.9pt
\scriptspace=0.9pt 
\arraycolsep0.9em 
e^n_1&=\frac{\hat{e}^n_1 \tau_1}{\gamma_2} \chi([0,\gamma_2])+\frac{\hat{e}^n_1(L-\tau_1)}{L-\gamma_2} \chi([\gamma_2,L]),\\
e^n_j&=\frac{\hat{e}^{1,n}_j\tau_j}{\gamma_1} \chi([0,\gamma_1])+ \left(\frac{\hat{e}^{1,n}_j(\gamma_2-\tau_j)}{\gamma_2-\gamma_1} + \frac{\hat{e}^{2,n}_j(\tau_j-\gamma_1)}{\gamma_2-\gamma_1}\right)\chi([\gamma_1,\gamma_2]) +\frac{\hat{e}^{2,n}_j(L-\tau_j)}{L-\gamma_2} \chi([\gamma_2,L]),\nonumber\\
e^n_{N}&=\frac{\hat{e}^n_N \tau_{N}}{\gamma_1} \chi([0,\gamma_1])+\frac{\hat{e}^n_{N}(L-\tau_{N})}{L-\gamma_1} \chi([\gamma_1,L]),\nonumber
\end{align} 
$j=2,\dots,N$,
where the unknown coefficients collected into $\mathbf{e}^n:=(e^n_1,e_2^{1,n},e_2^{2,n},\dots,e_N^n)^\top\in \mathbb{R}^{\tilde{N}}$, $\tilde{N}:=2(N-2)+2$ represent the value of the error functions at the traces, while $\chi([a,b])$ are characteristic functions which satisfy $\chi(\tau)=1$ if $\tau\in[a,b]$ and zero otherwise.
Inserting these expressions into the transmission conditions \eqref{eq:continuity_condition}-\eqref{eq:flux_condition}, we aim to express  $\hat{e}^{i,n}_j$ in terms of the coefficients of the errors in fractures $j-1$ and $j+1$ at iteration $n-1$. A direct calculation, which we omit due to space limitation (see \cite{vanzan_thesis} for more details) leads to the recurrence relation $\mathbf{e}^n=T^D_{N}\mathbf{e}^{n-1}=M_{N}^{-1}N_{N}\mathbf{v}^{n-1}$,
where $M_N, N_N\in \mathbb{R}^{\widetilde{N},\widetilde{N}}$ have the block structure
\scriptsize{\begin{align}\label{eq:M&N}
\setlength\arraycolsep{0.5pt}
M_{N}:=\begin{pmatrix}
F_1\\
& F_2\\
& & F_2\\
&&&& \ddots\\
&&&&& F_2\\
&&&&&& F_4
\end{pmatrix},\quad
N_{N}:=\begin{pmatrix}
& a & b\\
d_2\\
& & & a &b\\
& b & c\\ 
& & & & & \ddots &\ddots\\
& & & b & c & & &\\
& & & &\ddots & \ddots & a & b\\
& & & & & & & & & d_1\\
& & & & & & b & c \\
\end{pmatrix}
\end{align}}\normalsize
with blocks 
\footnotesize{\begin{align}
\setlength\arraycolsep{0.5pt}
F_1&:=p+\frac{L}{\gamma_2(L-\gamma_2)},\quad  F_2:=\begin{pmatrix}
p+\frac{\gamma_2}{\gamma_1(\gamma_2-\gamma_1)} & -\frac{1}{\gamma_2-\gamma_1}\\
-\frac{1}{\gamma_2-\gamma_1} & p+ \frac{L-\gamma_1}{(L-\gamma_2)(\gamma_2-\gamma_1)}\end{pmatrix} \quad F_4:=p+\frac{L}{\gamma_1(L-\gamma_1)},\\
a&:=p-\frac{\gamma_2}{\gamma_1(\gamma_2-\gamma_1)},\quad b:=\frac{1}{\gamma_2-\gamma_1},\quad c:=p-\frac{L-\gamma_1}{(L-\gamma_2)(\gamma_2-\gamma_1)},\quad d_j:=p-\frac{L}{\gamma_j(L-\gamma_j)}.
\end{align}
}\normalsize
The next theorem shows that the spectral radius of $T^D_N$ is bounded strictly below $1$ for every $N$ if Dirichlet b.c. are imposed on each fracture. Thus, the number of iterations to reach a given tolerance is indipendent of $N$, and the OSM is weakly scalable.
\begin{theorem}\label{thm:scalability_1D}
Let $\gamma_1+\gamma_2=L$ and $s_{j}^{+,-}=p$, $\forall j$. Then, the OSM is weakly scalable for the solution of problem \eqref{DFN} with Dirichlet b.c. on each $F_i$, in the sense that $\rho(T^D_{N})\leq C<1$, independently of $N$ for every $p>0$.
\end{theorem}
\begin{proof}
Notice that $ \rho(T^D_{N})=\rho(M^{-1}_{N}N_{N})=\rho(N_{N}M^{-1}_{N})\leq \|N_{N}M^{-1}_{N}\|_{\infty}$.
Direct calculations show that
\[ \|N_{N}M^{-1}_{N}\|_{\infty} = \max\left\{\left|\frac{p\gamma_2(L-\gamma_2)-L}{p\gamma_2(L-\gamma_2)+L}\right|,\frac{2p(L-\gamma_2)^2+\left|L+(L-2\gamma_2)(L-\gamma_2)^2p^2\right|}{(p(L-\gamma_2)+1)(p(L-\gamma_2)(2\gamma_2-L)+L)}\right\}.\]
The first term is clearly less than 1 for every $p>0$. For the second term, we distinguish two cases: if $L+(L-2\gamma_2)(L-\gamma_2)^2p^2<0$, then it simplifies to $\left|\frac {-1+ \left( L-\gamma_2 \right) p}{1+ \left( L-\gamma_2 \right) p}\right|$ which strictly less than 1.
Similarly, if $L+(L-2\gamma_2)(L-\gamma_2)^2 p^2\geq 0$, then $\frac{2p(L-\gamma_2)^2+\left|L+(L-2\gamma_2)(L-\gamma_2)^2p^2\right|}{(p(L-\gamma_2)+1)(p(L-\gamma_2)(2\gamma_2-L)+L)}=\left|{\frac {p \left( L-\gamma_2 \right)  \left( 2\gamma_2-L \right) -L}{
p \left( L-\gamma_2 \right)  \left(2\,\gamma_2-L \right) +L}}\right|<1$ being $2\gamma_2>L$.
Thus $\exists C<1$ indipendent on $N$ such that $\|N_{N}M^{-1}_{N}\|_{\infty}<C$ for every $p>0$.
\end{proof}
The hypothesis $\gamma_1+\gamma_2=L$ is used to simplify the otherwise coumbersome calculations, but it has not been observed in numerical experiments.

We emphasize that OSMs are not scalable for one-dimensional chains of fixed size-subdomains \cite{Chaouqui2018}. In our setting, the scalability is due to the geometrical configuration typical of DFNs, which permits to impose Dirichlet b.c. on each fracture, being the transmission conditions imposed in the interior. Thus, we observe error contraction before information is propagated through the iterations across the subdomains (see \cite[Section 3]{Chaouqui2018}).
With a similar argument, we expect the OSM not to be scalable if Neumann b.c. are applied on each fracture, as the errors in the middle fractures would require about $N/2$ to start contracting. To verify this, we can perform the same analysis by replacing \eqref{eq:expressions} with appropriate subdomains solutions. We then obtain the recurrence relation $\mathbf{e}^n=T^N_{N}\mathbf{e}^{n-1}=\widetilde{M}_{N}^{-1}\widetilde{N}_{N}\mathbf{v}^{n-1}$, where $\widetilde{M}_{N}$ $\widetilde{N}_{N}$ have the same structure of \eqref{eq:M&N}, but with blocks
\begin{equation*}
\medmuskip=-1mu
\thinmuskip=-1mu
\thickmuskip=-1mu
\nulldelimiterspace=0.9pt
\scriptspace=0.9pt 
\arraycolsep0.9em 
\begin{aligned}
\widetilde{F}_1&:=p+\frac{1}{\gamma_2},\quad \widetilde{F}_4:=p+\frac{1}{L-\gamma_1},\quad 
 \widetilde{F}_2:=\begin{pmatrix}
p+\frac{1}{\gamma_2-\gamma_1} & -\frac{1}{\gamma_2-\gamma_1}\\
-\frac{1}{\gamma_2-\gamma_1} & p+ \frac{1}{(\gamma_2-\gamma_1)}\end{pmatrix},\\
\widetilde{a}&:=p-\frac{1}{\gamma_2-\gamma_1},\quad \widetilde{b}:=\frac{1}{\gamma_2-\gamma_1},\quad \widetilde{c}:=\widetilde{a},\quad \widetilde{d}_j:=p-\frac{1}{(L-\gamma_j)}.
\end{aligned}
\end{equation*}
\begin{figure}
\centering
\includegraphics[scale=0.25]{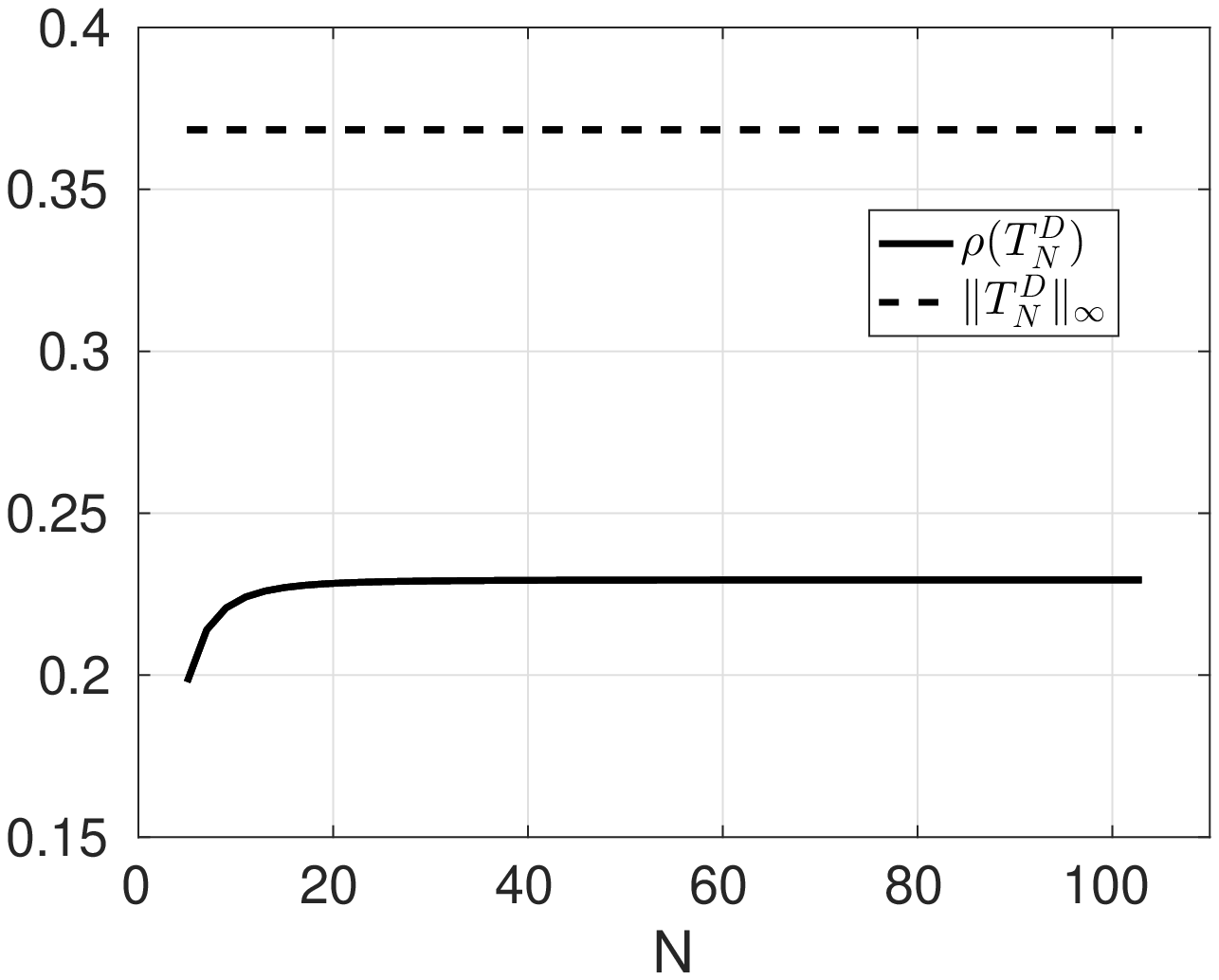}
\includegraphics[scale=0.25]{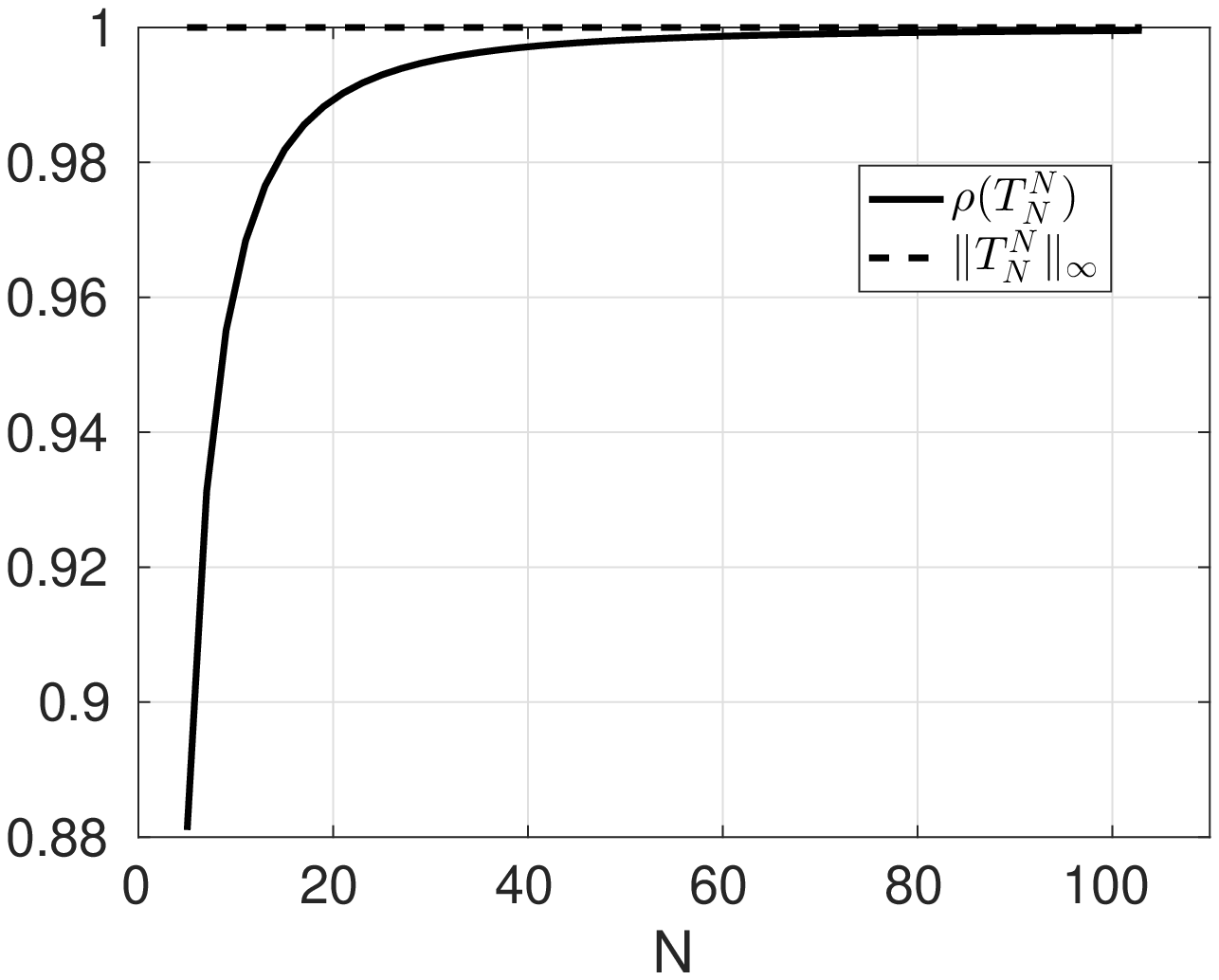}
\includegraphics[scale=0.25]{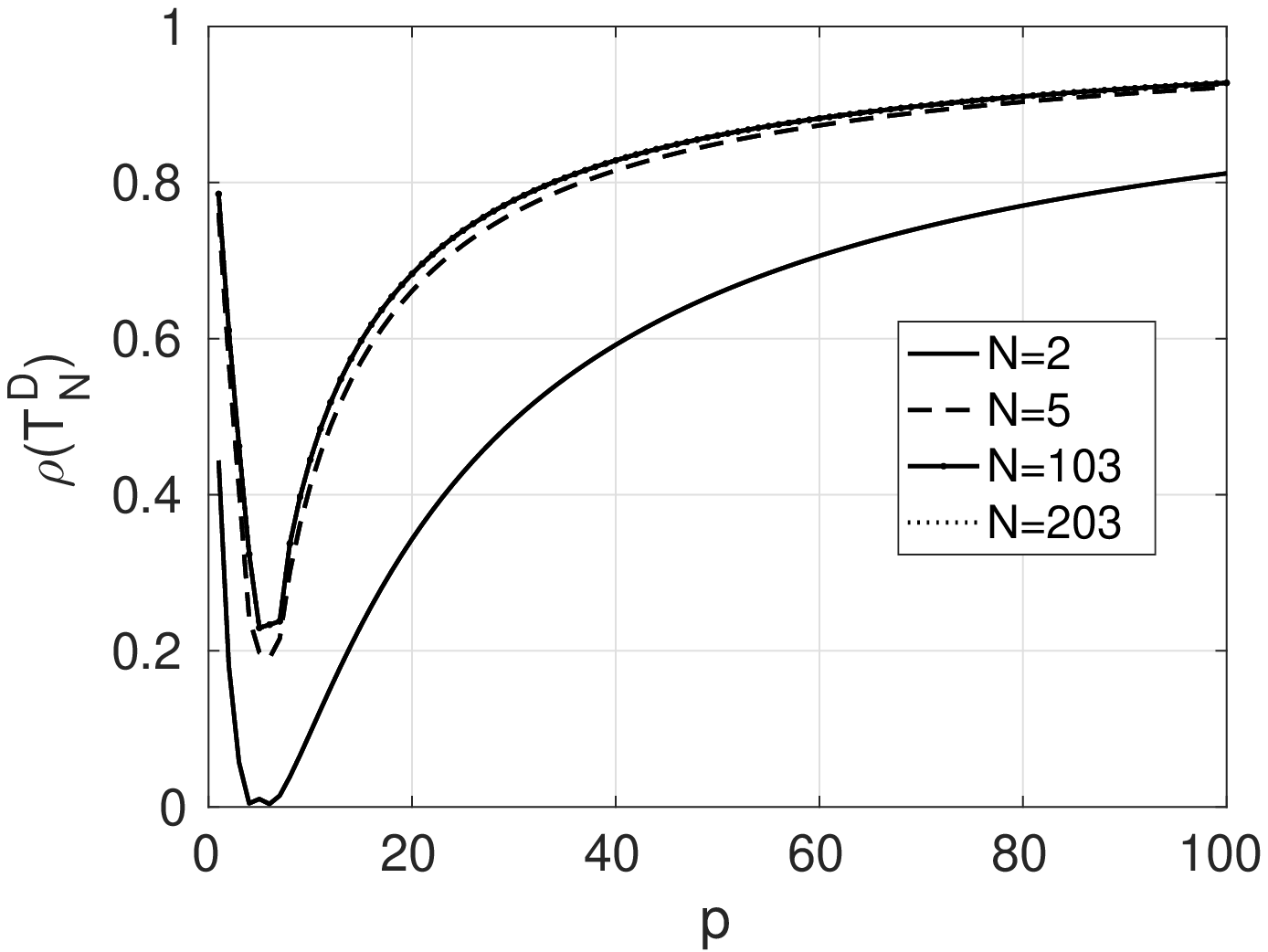}
\caption{Left and center panel: spectral radii of $T^D_{N}$ and $T^N_{N}$ as the number of fractures increases. Right panel: spectral radius of $T^D_{N}$ as $p$ varies. Parameters: $L=1$, $\gamma_1=0.2$, $\gamma_2=0.6$, $\nu=1$.}\label{Fig:scalability}
\end{figure}
The first two panels of Fig. \ref{Fig:scalability} show the dependence of the spectral radii of $T^D_{N}$ and $T^N_{N}$ as $N$ increases. While $\rho(T^D_{N})$ remains bounded below one, $\rho(T^N_{N})$ tends rapidly to one as $N$ grows, thus the OSM is not weakly scalable if Neumann b.c. are used.

We remark that in applications it is quite common to impose homogeneous Neumann b.c. in internal fractures because at the tip of the fracture the flow exchange with the surrouding matrix is negligible.
In such cases, the analysis suggests two possible heuristics to improve the convergence of DD solvers. The first one is to stress the importance of an efficient partition of the fractures into subdomains (each subdomain generally contains more than one fracture). Such partition should minimize the maximum, over floating subdomains, of the distance of each subdomain from the Dirichlet boundary $\Gamma_d$ (see \cite{vanzan_thesis} for numerical experiments). Recall that a subdomain $\Omega_j$ is called ``floating subdomain" if $\partial \Omega_j \cap \Gamma_D=\emptyset$. 
The second heuristic is to replace Neumann b.c. with Robin ones (which would also model the realistic case of a flux across $\partial F_j$).
Ref \cite{ciaramella2021effect} showed that Robin b.c. permits to recover scalability of OSM as in the Dirichlet case.
 
Notice that the rate of convergence of OSMs, which may be indipendent of $N$ (see discussion above), still depends on the transmission conditions, hence it is important to have good estimates of the parameters $s_{j}^{+,-}$.
To estimate them, we consider two fractures $F_1$ and $F_2$, which are coupled across a single trace. The general solutions are given by
\begin{align}
\medmuskip=-1mu
\thinmuskip=-1mu
\thickmuskip=-1mu
\nulldelimiterspace=0.9pt
\scriptspace=0.9pt 
\arraycolsep0.9em 
e^n_1&=\frac{\hat{e}^n_1 \tau_1}{\gamma_2} \chi([0,\gamma_2])+\frac{\hat{e}^n_1(L-\tau_1)}{L-\gamma_2} \chi([\gamma_2,L]),\quad 
e^n_2&=\frac{\hat{e}^n_2 \tau_1}{\gamma_1} \chi([0,\gamma_1])+\frac{\hat{e}^n_2(L-\tau_1)}{L-\gamma_1} \chi([\gamma_1,L]),
\end{align}
where the unknowns are the two coefficients $\hat{e}^n_1$ and $\hat{e}^n_2$.
Inserting these solutions in the transmission conditions we obtain the scalar recurrence relation 
\[\hat{e}^n_j=\rho_{1D}(s^-_1,s^+_1,\nu_1,\nu_2)\hat{e}^{n-2}_j,\ j=1,2,\quad  \rho_{1D}(s^-_1,s^+_1,\nu_1,\nu_2)=\frac{\left(\frac{\nu_2L}{\gamma_1(L-\gamma_1)}-s^-_1\right)\left(\frac{\nu_1L}{\gamma_2(L-\gamma_2)}-s^+_1\right)}{\left(\frac{\nu_1L}{\gamma_2(L-\gamma_2)}+s^-_1\right)\left(\frac{\nu_2L}{\gamma_1(L-\gamma_1)}+s^+_1\right)}.\]
If we chose $s_1^-=s^{-,\opt}_1:=\frac{\nu_2L}{\gamma_1(L-\gamma_1)}$ and $s_1^+=s^{+,\opt}_1:=\frac{\nu_1L}{\gamma_2(L-\gamma_2)}$, we would have $\rho(s^-_1,s^+_1,\nu_1,\nu_2)=0$, that is the OSM is nilpotent. The right panel of Fig. \ref{Fig:scalability} verifies that the two fracture analysis provides very good estimates for the optimal Robin parameters in the many-fractures case.

\section{Scalability analysis for two-dimensional DFNs}

In this section we consider the two dimensional extension of Fig. \ref{fig:1Dstair}. Each fracture $F_j$ is a two dimensional polygon, see Fig. \ref{DFN2D}, and the traces, denoted by $S_j$, are straight segments crossing the whole fracture.
On each fracture, the local reference system has coordinates $\{\tau_1,\tau_2\}$.
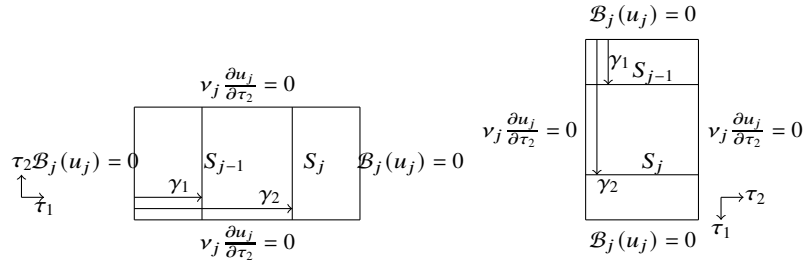
\begin{figure}
\centering
\begin{tikzpicture}[scale=0.30]
\draw[black] (-5,5) -- (5,5);
\draw[black] (5,5) -- (5,0);
\draw[black] (5,0) -- (-5,0);
\draw[black] (-5,0) -- (-5,5);
\draw[black] (-2,0) -- (-2,5);
\draw[black] (2,0) -- (2,5);
\node at (0,6) {$\nu_j\frac{\partial u_j}{\partial \tau_2}=0$};
\node at (0,-1) {$\nu_j\frac{\partial u_j}{\partial \tau_2}=0$};
\node at (-7.2,2.5) {$\mathcal{B}_j( u_j)=0$};
\node at (7.2,2.5) {$\mathcal{B}_j( u_j)=0$};
\node at (-1,2.5) {$S_{j-1}$};
\draw[->] (-5,1) -- (-2,1);
\node at (-3,1.5) {$\gamma_1$};
\draw[->] (-5,0.5) -- (2,0.5);
\node at (1,1) {$\gamma_2$};
\node at (3,2.5) {$S_j$};
\draw[->] (-10,1) -- (-9,1);
\node at (-9,0.5) {$\tau_1$};
\draw[->] (-10,1) -- (-10,2);
\node at (-10,2.5) {$\tau_2$};

\draw[black] (15,8) -- (15,0);
\draw[black] (15,8) -- (20,8);
\draw[black] (15,0) -- (20,0);
\draw[black] (20,0) -- (20,8);
\node at (22.5,4) {$\nu_j\frac{\partial u_j}{\partial \tau_2}=0$};
\node at (12.5,4) {$\nu_j\frac{\partial u_j}{\partial \tau_2}=0$};
\node at (17.5,9) {$\mathcal{B}_j( u_j)=0$};
\node at (17.5,-1) {$\mathcal{B}_j( u_j)=0$};
\node at (18,6.5) {$S_{j-1}$};
\draw[black] (15,6) -- (20,6);
\draw[black] (15,2) -- (20,2);
\node at (18,2.5) {$S_{j}$};

\draw[->] (16,8) -- (16,6);
\node at (16.5,7) {$\gamma_1$};
\draw[->] (15.5,8) -- (15.5,2);
\node at (16,1.5) {$\gamma_2$};

\draw[->] (21,1) -- (22,1);
\node at (22.5,1) {$\tau_2$};
\draw[->] (21,1) -- (21,0);
\node at (21,-0.5) {$\tau_1$};
\end{tikzpicture}
\caption{Geometry of a two dimensional fracture.}\label{DFN2D}
\end{figure}
Due to the geometrical configuration, the error can be expanded in Fourier series in each fracture, i.e.
$e_j=\sum_{k=0}^{\infty} \tilde{e}_j(\tau_1,k)\cos(\frac{k\pi}{L}\tau_2)$. The Fourier coefficients $\tilde{e}_j(\tau_1,k)$ are obtained imposing the b.c. and the transmission conditions.
The long expressions are omitted due to space limitation (see for complete expressions \cite{vanzan_thesis}). We only report the expressions for the first subdomain
\begin{align}
\tilde{e}^n_1(\tau_1,k)&=\hat{e}^n_1(k)\frac{\sinh(\frac{k\pi}{L}\tau_1)}{\sinh(\frac{k\pi}{L}\gamma_2)}\chi([0,\gamma_2])+ \hat{e}^n_1(k) \frac{\sinh(\frac{k\pi}{L}(L-\tau_1))}{\sinh(\frac{k\pi}{L}(L-\gamma_2))}\chi([\gamma_2,L]),\ k>0,\\
\tilde{e}^n_1(\tau_1,0)&=\frac{\hat{e}^n_1(0) \tau_1}{\gamma_2} \chi([0,\gamma_2])+\frac{\hat{e}^n_1(0)(L-\tau_1)}{L-\gamma_2} \chi([\gamma_2,L]),\ k=0.
\end{align} 
The unknowns $\hat{e}^{i,n}_j(k)$ are the values attained by the $k$-th mode of the Fourier expansions at each trace. In numerical computations, $k\in [k_{\min},k_{\max}]$ for Dirichlet b.c. while $k\in [0,k_{\max}]$ for Neumann b.c., $k_{\max}=\frac{\pi}{h}$ being the maximum frequency supported by the numerical grid and $k_{\min}=\frac{\pi}{L}$.
Similarly to the 1D case, one can obtain recurrence relations which link the Fourier coefficients of one fracture at iteration $n$ as functions of the Fourier coefficients of the neighbouring fractures at iteration $n-1$.
In particular for $k=0$, $\mathbf{v}^n_0:=\left(\hat{e}^n_1(0),\hat{e}^{1,n}_2(0),\hat{e}^{2,n}_2(0),\dots,\hat{e}^n_N(0)\right)$ satisfies $\mathbf{v}^n_0=T^D_{N}\mathbf{v}^{n-1}_0$, where $T^D_{N}$ is the matrix of the 1D system with Dirichlet b.c..
For $k>0$, we obtain instead $\mathbf{v}^n_k=T^{2D}_N(k)\mathbf{v}^{n-1}_k$, where
$T^{2D}_N=M^{-1}_{2D}N_{2D}$ has the same block structure of the 1D case but with blocks defined as  
\[F_2:=\begin{pmatrix}
p+\coth(\frac{k\pi}{L}\gamma_1)+\coth(\frac{k\pi}{L}(\gamma_2-\gamma_1)) & -\frac{1}{\coth(\frac{k\pi}{L}(\gamma_2-\gamma_1)}\\
-\frac{1}{\sinh(\frac{k\pi}{L}(\gamma_2-\gamma_1)} & p+\coth(\frac{k\pi}{L}(L-\gamma_2))+\coth(\frac{k\pi}{L}(\gamma_2-\gamma_1))
\end{pmatrix},\]
$F_1:=p+\coth(\frac{k\pi}{L}\gamma_2)+\coth(\frac{k\pi}{L}(L-\gamma_2))$ and $F_4:=p+\coth(\frac{k\pi}{L}(L-\gamma_1))+\coth(\frac{k\pi}{L}(\gamma_1))$.
On the other hand the coefficients of $N_{2D}$ are 
\begin{align*}
\medmuskip=-1mu
\thinmuskip=-1mu
\thickmuskip=-1mu
\nulldelimiterspace=0.9pt
\scriptspace=0.9pt 
\arraycolsep0.9em 
a&:=p-\coth\left(\frac{k\pi}{L}\gamma_1\right)-\coth\left(\frac{k\pi}{L}(\gamma_2-\gamma_1)\right),\quad b:=\frac{1}{\sinh(\frac{k\pi}{L}(\gamma_2-\gamma_1)},\\
c&:=p-\coth\left(\frac{k\pi}{L}(L-\gamma_2)\right)-\coth\left(\frac{k\pi}{L}(\gamma_2-\gamma_1)\right),\quad d_j:=p-\coth\left(\frac{k\pi}{L}(L-\gamma_j)\right)-\coth\left(\frac{k\pi}{L}\gamma_j\right).
\end{align*}
Fig \ref{Fig:scalability2D} shows numerically that the OSM is scalable also for a 2D DFN with Dirichlet b.c.. Observing that the frequency $k=0$ behaves according to the 1D analysis, we expect the OSM with Neumann b.c. on each fracture except on the first and last ones not to be weakly scalable. Repeating the calculations one finds an iteration matrix $\widetilde{T}^{2D}_N$ and Fig. \ref{Fig:scalability2D} confirms this conclusion.
\begin{figure}
\centering
\includegraphics[scale=0.255]{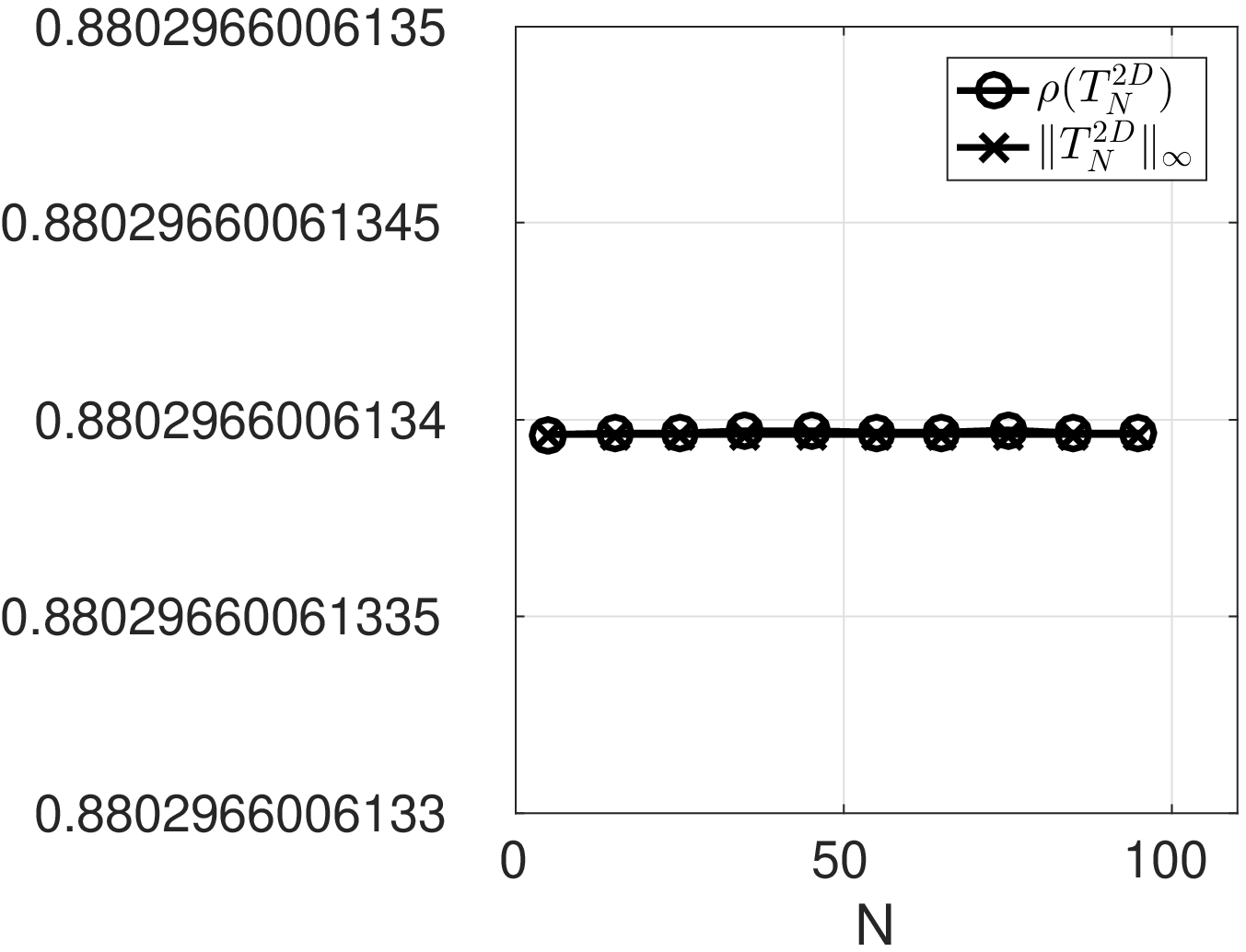}
\includegraphics[scale=0.255]{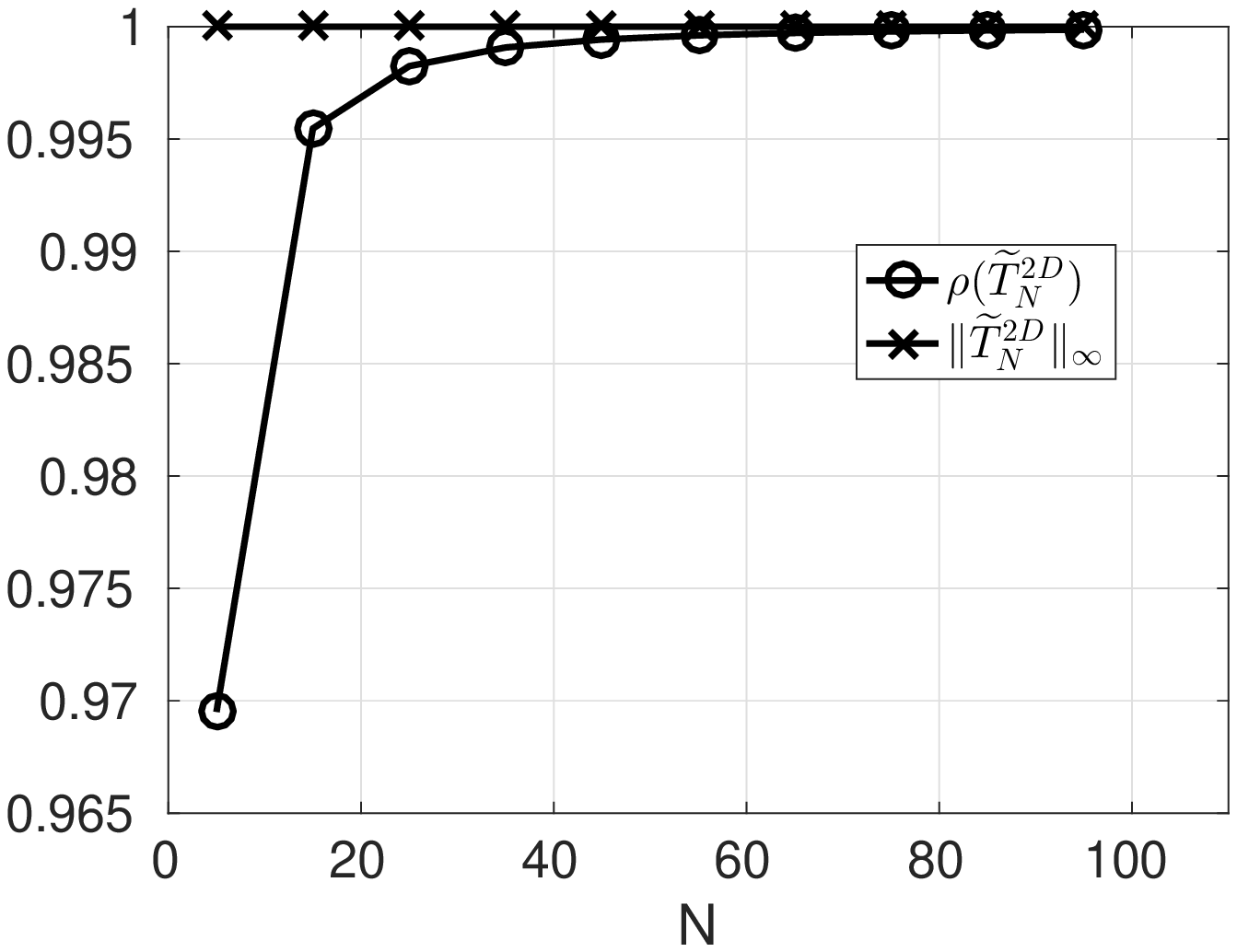}
\includegraphics[scale=0.255]{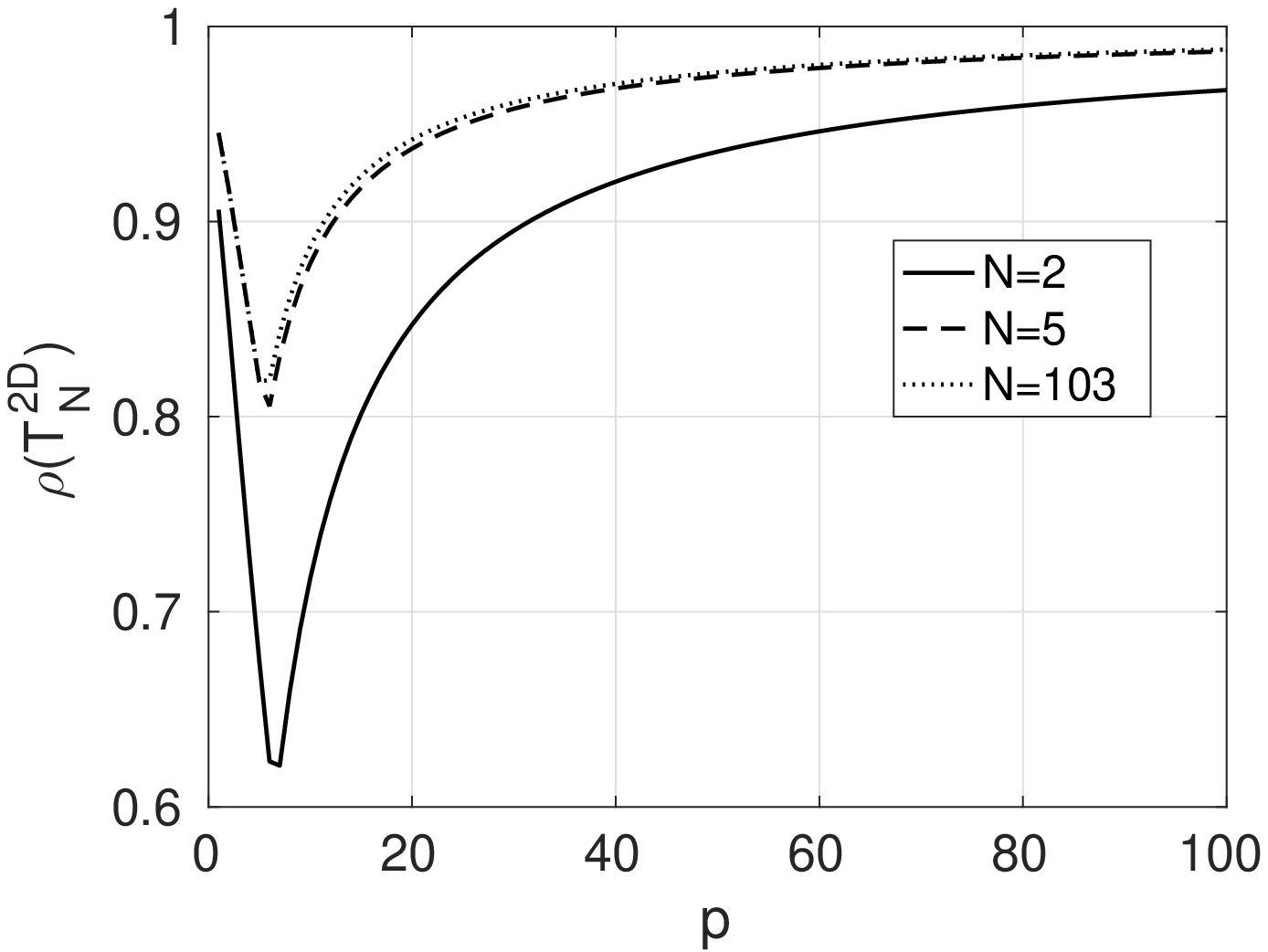}
\caption{Left and center panel: spectral radii of $\max_{k\in [k_{\min},k_{\max}}T_N^{2D}(k)$ and $\max_{k\in (0,k_{\max})]}\widetilde{T}_N^{2D}(k)$ as $N$ grows. Parameters: $L=1$, $\gamma_1=0.2$,$\gamma_2=0.6$ and $p=20$. Right panel: $\max_{k\in [k_{\min},k_{\max}]}T_N^{2D}(k)$ as $p$ varies.}\label{Fig:scalability2D}
\end{figure}

We now  derive the optimized parameters by analyzing the coupling of two fractures. Inserting the Fourier expansions into the transmission conditions and defining
\begin{equation*}
\medmuskip=-1mu
\thinmuskip=-1mu
\thickmuskip=-1mu
\nulldelimiterspace=0.9pt
\scriptspace=0.9pt 
\arraycolsep0.9em 
f_1(k):=\frac{\nu_2k\pi}{L}\left(\coth\left(\frac{k\pi}{L}\gamma_1\right)+\coth\left(\frac{k\pi}{L}(L-\gamma_1)\right)\right),\quad f_2(k):=\frac{\nu_1k\pi}{L}\left(\coth\left(\frac{k\pi}{L}\gamma_2\right)+\coth\left(\frac{k\pi}{L}(L-\gamma_2)\right)\right),
\end{equation*}
we obtain $\hat{e}^n_j(k)=\rho(k,s^-_1,s^+_1)\hat{e}^{n-2}_j(k)$, for $k>0$,$j=1,2$, where $\rho(k,s^-_1,s^+_1):=\frac{f_1(k)-s^-_1}{f_2(k)+s^-_1}\frac{f_2(k)-s^+_1}{f_1(k)+s^+_1}$.
On the other hand, for the constant mode $k=0$ we recover the 1D result: $\hat{e}^n_j(0)=\rho_{1D}(s^-_1,s^+_1)\hat{e}^{n-2}_j(0)$.
To derive optimized parameters, we set $s^-_1=f_1(p)$, $s^+_1=f_2(p)$ for some $p\in \mathbb{R}^+$, and we study
\begin{equation}\label{eq:minmax2}
\min_{p \in \mathbb{R}^+}\max\left\{\rho_{1D}(p),\max_{k\in [\frac{\pi}{L},k_{\max}]} \rho(k,p)\right\}.
\end{equation}
Despite $\rho(k,p)$ is not defined at $k=0$, since $\coth(\cdot)$ has a singularity, we observe that
$\lim_{k\rightarrow 0} \rho(k,p)= \rho_{1D}(p)$.
Thus, we introduce the function $\widetilde{\rho}(k,p)=\rho(k,p)$ for $k>0$ and $\widetilde{\rho}(0,p)=\rho_{1D}(p)$, and further simplify the min-max problem to 
\begin{equation}\label{eq:minmax3}
\min_{p \in \mathbb{R}^+}\max_{k\in [0,k_{\max}]} \widetilde{\rho}(k,p).
\end{equation}
The next theorem can be proved using the very same steps of \cite[Theorem 2.3]{gander2019heterogeneous}. Fig. \ref{Fig:scalability2D} confirms that effectiveness of the analysis even in the many-fractures case.

\begin{theorem}\label{theorem1}
  The solution of the min-max problem \eqref{eq:minmax3} is given by the
  unique $p^*$ which satisfies $\widetilde{\rho}(0,p)=\widetilde{\rho}(k_{\max},p)$.
\end{theorem}
Future works will focus on testing the results of the analysis presented on more realistic DFN configurations.
\vspace*{-0.2cm}
\bibliographystyle{spmpsci.bst}
\bibliography{references.bib}
\end{document}